\documentclass[12pt,leqno, draft]{amsart}

\usepackage{amssymb}
\usepackage{amsthm}
\usepackage{amsmath}
\usepackage{amsfonts}
\usepackage[mathscr]{eucal}
\usepackage{enumerate}
\theoremstyle{plain}
\newtheorem{Thm}{Theorem}[section]
\newtheorem{Cor}{Corollary}[section]

\newtheorem{Rem}{Remark}
\newtheorem{Opp}{Open problem}

\theoremstyle{definition}
\newtheorem{Def}{Definition}[section]
\theoremstyle{remark}
\numberwithin{equation}{section}

\begin{document}

\title[Recent results on the stability of an equation]
{Recent results on the stability of the parametric fundamental equation of information}
\author{Eszter Gselmann}

\address{
Institute of Mathematics\\
University of Debrecen\\
P.~O.~Box: 12.\\
Debrecen\\
Hungary\\
H--4010}
\email{gselmann@math.klte.hu}

\thanks{This research has been supported by the Hungarian Scientific Research Fund
(OTKA) Grant NK 68040.}

\begin{abstract}
The purpose of this paper is to summarize the
recent results on the stability of the parametric
fundamental equation of information.
Furthermore, by the help of a modification of a method we used in
\cite{GM08} we shall give a unified proof for the
Hyers--Ulam stability of the equation
in question, assuming that the parameter does not equal to 1.
As a corollary of the main result, a system of equations,
that defines the recursive and semi--symmetric
information measures is also discussed.
\end{abstract}

\maketitle

\section{Introduction and preliminaries}

The study of stability problems for functional equations
originates from a famous question of Ulam.
In his talk he asked whether it is true that the
solution of an equation differing slightly from a
given one, must of necessity be close to the solution
of this equation (see Ulam \cite{Ula64} page 63).
Concerning the additive Cauchy equation,
Hyers gave an affirmative answer to Ulam's question in
1941 (see Hyers \cite{Hye41}).
Since then, this result has been extended and generalized in several
ways (see e.g. Forti \cite{For95}, Ger \cite{Ger94}, Hyers--Isac--Rassias \cite{HIR98} and
Moszner \cite{Mos04}), and the stability theory has become a dynamically
developing field of research.

In this paper the previous problem is investigated concerning
the parametric fundamental equation of information, i.e., equation
\begin{equation}\label{Eq1.1}
f(x)+(1-x)^{\alpha}f\left(\frac{y}{1-x}\right)=
f(y)+(1-y)^{\alpha}f\left(\frac{x}{1-y}\right).
\end{equation}
If $\alpha=1$, then equation \eqref{Eq1.1} is called the fundamental
equation of information (see Acz\'{e}l--Dar\'{o}czy \cite{AD75}).
However, before this, we shall fix the notation and the terminology
that will be used throughout this paper.
As usual, $\mathbb{R}$  denotes the set of the real numbers
and on $\mathbb{R}_{+}$ and on $\mathbb{R}_{++}$
we understand the set of the nonnegative and the positive real numbers, respectively.
Furthermore, let $n$ be a fixed positive integer
and define the following sets
\[
\Gamma_{n}=\left\{(p_{1}, \ldots, p_{n})\in\mathbb{R}^{n}
\vert p_{i}\geq 0, i=1, \ldots, n, \sum^{n}_{i=1}p_{i}=1\right\},
\]

\[
\Gamma^{\circ}_{n}=\left\{(p_{1}, \ldots, p_{n})\in\mathbb{R}^{n}
\vert p_{i}> 0, i=1, \ldots, n, \sum^{n}_{i=1}p_{i}=1\right\},
\]

\[
D=\left\{(x, y)\in\mathbb{R}^{2}\vert x, y\in [0, 1[, x+y\leq 1\right\}
\]
and

\[
D^{\circ}=\left\{(x, y)\in\mathbb{R}^{2}\vert x, y, x+y\in ]0, 1[\right\}.
\]

To make our result comprehensible, first we list some basic facts
from the theory of functional equations. These can be
found e.g. in Kuczma \cite{Kuc85} and in Rad\'{o}--Baker \cite{RB87}.

\begin{Def} \cite{Kuc85}, \cite{RB87}
Let $I\subset\mathbb{R}_{+}$ and
\[
\mathcal{A}=\left\{(x, y)\in\mathbb{R}^{2}_{+}\vert x, y, x+y\in I\right\}.
\]
A function
$a:I\rightarrow\mathbb{R}$ is called \emph{additive on $\mathcal{A}$} if
\begin{equation}\label{Eq1.2}
a\left(x+y\right)=a\left(x\right)+a\left(y\right)
\end{equation}
holds for all pairs $(x, y)\in \mathcal{A}$. \\
\noindent
Consider the set
\[
\mathcal{I}=\left\{(x, y)\in\mathbb{R}^{2}_{+}\vert x, y, xy \in I\right\}.
\]
We say that
$\mu:I\rightarrow\mathbb{R}$ is \emph{multiplicative on $\mathcal{I}$} if the functional equation
\begin{equation}\label{Eq1.3}
\mu\left(x y\right)=\mu\left(x\right)\mu\left(y\right)
\end{equation}
is fulfilled for all $(x, y)\in\mathcal{I}$. \\
\noindent
If
\[\mathcal{L}=\left\{(x, y)\in\mathbb{R}^{2}_{++}\vert x, y, xy \in I\right\}
\]
then a function
$l:I\rightarrow\mathbb{R}$ is called \emph{logarithmic on $\mathcal{L}$}
if it satisfies the functional equation
\begin{equation}\label{Eq1.4}
l\left(x y\right)=l\left(x\right)+l\left(y\right)
\end{equation}
for all $(x, y)\in \mathcal{L}$.
\end{Def}

The parametric fundamental equation of information arises
in a natural way in the characterization problem
of information measures.
A sequence $(I_{n})$ of real-valued functions
on $\Gamma^{\circ}_{n}$ or on $\Gamma_{n}$ is
called an \emph{information measure} on the open or on the closed
domain, respectively.
The usual information-theoretical interpretation is that
$I_{n}(p_{1}, \ldots, p_{n})$ is a measure
of uncertainty as to the outcome of an  experiment
having $n$ possible outcomes with probabilities $p_{1}, \ldots, p_{n}$,
or, in other words, it is the amount of information received
from the knowledge of which of the possible outcomes occurred.

Some desiderata for information measures can be found in
Acz\'{e}l--Dar\'{o}czy \cite{AD75} as well as in
Ebanks--Sahoo--Sander \cite{ESS98}.
Nevertheless, in this paper we will use only the following
properties. The reader should consult
Acz\'{e}l \cite{Acz81}, \cite{Acz86},
Dar\'{o}czy \cite{Dar70}, Havrda--Charv\'{a}t \cite{HC67} and Tsallis \cite{Tsa88}, as well.
\begin{Def}
The sequence of functions $I_{n}:\Gamma^{\circ}_{n}\rightarrow\mathbb{R}$
($n=2, 3, \ldots$) is
\begin{enumerate}[(i)]
	\item \emph{$\alpha$--recursive}, if
	\[
	I_{n}\left(p_{1}, \ldots, p_{n}\right)=I_{n-1}\left(p_{1}+p_{2}, p_{3}, \ldots, p_{n}\right)+
	\left(p_{1}+p_{2}\right)^{\alpha}I_{2}\left(\frac{p_{1}}{p_{1}+p_{2}}, \frac{p_{2}}{p_{1}+p_{2}}\right)
\]
holds for all $n=3, 4, \ldots$ and $\left(p_{1}, \ldots, p_{n}\right)\in\Gamma^{\circ}_{n}$, with some
$\alpha\in\mathbb{R}$.
\item \emph{3--semi--symmetric}, if
	\[
	I_{3}\left(p_{1}, p_{2}, p_{3}\right)=I_{3}\left(p_{1}, p_{3}, p_{2}\right)
\]
holds for all $\left(p_{1}, p_{2}, p_{3}\right)\in\Gamma^{\circ}_{3}$.
\end{enumerate}
\end{Def}
Measures depending on one probability distribution are generally
referred as \emph{entropies}.
Probably the most well-known of all is the \emph{Shannon-entropy}
\[
H^{1}_{n}(p_{1}, \ldots, p_{n})=
-\sum^{n}_{i=1}p_{i}\log_{2}\left(p_{i}\right),
\quad \left((p_{1}, \ldots, p_{n})\in\Gamma^{\circ}_{n}\right)
\]
and \emph{the entropy of degree $\alpha$
(or the Havrda-Charv\'{a}t-entropy that recently has also been called Tsallis-entropy)}
\[
H^{\alpha}_{n}(p_{1}, \ldots, p_{n})=
\left(2^{1-\alpha}-1\right)^{-1}\left(\sum^{n}_{i=1}p^{\alpha}_{i}-1\right).
\quad \left(\alpha\neq 1, (p_{1}, \ldots, p_{n})\in\Gamma^{\circ}_{n}\right)
\]
It is easy to see that, for all $(p_{1}, \ldots, p_{n})\in\Gamma^{\circ}_{n}$,
\[
\lim_{\alpha\rightarrow 1}H^{\alpha}_{n}(p_{1}, \ldots, p_{n})=
H^{1}_{n}(p_{1}, \ldots, p_{n})
\]
holds and this shows that $(H_{n}^{1})$ can be continuously embedded to the family of $(H^{\alpha}_{n})$.

The following theorem enables us to transform the
characterization of information measures into solving
functional equations (see Acz\'{e}l--Dar\'{o}czy \cite{AD75} and Ebanks--Sahoo--Sander \cite{ESS98}).
\begin{Thm}
If the sequence of functions
$I_{n}:\Gamma^{\circ}_{n}\rightarrow\mathbb{R}$,
$(n=2, 3, \ldots)$ is $\alpha$--recursive and
$3$--semi--symmetric, then
the function $f:]0, 1[\rightarrow\mathbb{R}$ defined by
\[
f(x)=I_{2}(1-x, x)\quad \left(x\in ]0, 1[\right)
\]
satisfies functional equation (\ref{Eq1.1}) for all
$(x, y)\in D^{\circ}$.
\end{Thm}

\section{Known results}

In this section we will shortly list the results which
have been achieved  in the last academic year on the
stability of the parametric fundamental equation of information.

Concerning this topic, the first result was
the stability of equation (\ref{Eq1.1})  on the set $D$,
assuming that $1\neq \alpha>0$ (see Maksa \cite{Mak08}). Furthermore
the stability constant, he has got in that paper is much smaller than that of our.
However, the method, used in Maksa \cite{Mak08} does not work if
$\alpha=1$ or $\alpha\leq 0$ or if we consider the problem on the open domain.

After that, it was proved that equation (\ref{Eq1.1}) is stable in the sense
of Hyers and Ulam on the set $D^{\circ}$ as well as on $D$, assuming that $\alpha\leq 0$
(see \cite{GM08}). Recently it turned out that this method is appropriate
to prove superstability  in case $1\neq \alpha>0$. Thus we can give a unified proof for
the stability problem of equation (\ref{Eq1.1}) except the case $\alpha=1$.

\section{The main result}

Our main result is contained in the following theorem.

\begin{Thm}\label{Thm2.1}
Let $\alpha, \varepsilon\in\mathbb{R}$ be fixed,
$\alpha\neq 1, \varepsilon\geq 0$.
Suppose that the function $f:]0, 1[\rightarrow\mathbb{R}$
satisfies the inequality
\begin{equation}\label{Eq2.1}
\left|f(x)+(1-x)^{\alpha}f\left(\frac{y}{1-x}\right)
-f(y)-(1-y)^{\alpha}f\left(\frac{x}{1-y}\right)
\right|\leq\varepsilon
\end{equation}
for all $(x, y)\in D^{\circ}$.
Then, in case $\alpha=0$,
there exists a logarithmic function
$l:]0, 1[\rightarrow\mathbb{R}$ and $c\in\mathbb{R}$ such that
\begin{equation}\label{Eq2.2}
\left|f(x)-\left[l(1-x)+c\right]\right|\leq K(\alpha)\varepsilon,
\quad \left(x\in ]0, 1[\right)
\end{equation}
furthermore, if $\alpha\notin \left\{0, 1\right\}$, there
exist $a, b\in\mathbb{R}$ such that
\begin{equation}\label{Eq2.3}
\left|f(x)-\left[ax^{\alpha}+b(1-x)^{\alpha}-b\right]\right|\leq K(\alpha)\varepsilon
\end{equation}
holds for all $x\in ]0, 1[$, where
\[
K(\alpha)=\left\{
\begin{array}{lcl}
\left|2^{1-\alpha}-1\right|^{-1}\left(8+6\cdot2^{\alpha}+2^{-\alpha}\right), &\text{ if}& \alpha<0 \\
63, &\text{ if}& \alpha=0\\
\left|2^{1-\alpha}-1\right|^{-1}
\left(3+12\cdot2^{\alpha}+\frac{32\cdot3^{\alpha+1}}{\left|2^{-\alpha}-1\right|}\right), &\text{ if}& \alpha>0 .
\end{array}
\right.
\]
\end{Thm}

\begin{proof}
Define the function $F$ on $\mathbb{R}^{2}_{++}$ by
\begin{equation}\label{Eq2.4}
F(u, v)=(u+v)^{\alpha}f\left(\frac{v}{u+v}\right).
\end{equation}
Then
\begin{equation}\label{Eq2.5}
F(tu, tv)=t^{\alpha}F(u, v) \quad \left(t, u, v \in \mathbb{R}_{++}\right)
\end{equation}
and
\begin{equation}\label{Eq2.6}
f(x)=F(1-x, x), \quad \left(x\in ]0,1[\right)
\end{equation}
furthermore, with the substitutions
\[
x=\frac{w}{u+v+w}, \quad y=\frac{v}{u+v+w} \quad \left(u, v, w\in \mathbb{R}_{++}\right)
\]
inequality (\ref{Eq2.1}) implies that
\[
\begin{array}{l}
\left|f\left(\frac{w}{u+v+w}\right)+\frac{(u+v)^{\alpha}}{(u+v+w)^{\alpha}}f\left(\frac{v}{u+v}\right)\right. \\
\left. -f\left(\frac{v}{u+v+w}\right)-\frac{(u+w)^{\alpha}}{(u+v+w)^{\alpha}}f\left(\frac{w}{u+w}\right)\right|
\leq \varepsilon
\end{array}
\]
whence, by (\ref{Eq2.4})
\begin{equation}\label{Eq2.7}
\left|F(u+v, w)+F(u, v)-F(u+w, v)-F(u, w)\right|\leq \varepsilon (u+v+w)^{\alpha}
\end{equation}
follows for all $u, v, w\in\mathbb{R}_{++}$.

In the next step we define the functions $g$ and $G$ on
$\mathbb{R}_{++}$ and on $\mathbb{R}_{++}^{2}$, respectively by
\begin{equation}\label{Eq2.8}
g(u)=F(u, 1)-F(1, u)
\end{equation}
and
\begin{equation}\label{Eq2.9}
G(u, v)=F(u, v)+g(v).
\end{equation}
We will show that
\begin{equation}\label{Eq2.10}
\left|G(u, v)-G(v, u)\right|\leq 3\varepsilon (u+v+1)^{\alpha}.
\quad \left(u, v\in \mathbb{R}_{++}\right)
\end{equation}
Indeed, with the substitution $w=1$, inequality (\ref{Eq2.7})
implies that
\begin{equation}\label{Eq2.11}
\left|F(u+v, 1)+F(u, v)-F(u+1, v)-F(u, 1)\right|\leq \varepsilon (u+v+1)^{\alpha}.
\end{equation}
Interchanging $u$ and $v$, it follows from (\ref{Eq2.11}) that
\[
\left|-F(u+v, 1)-F(v, u)+F(v+1, u)-F(v, 1)\right|\leq \varepsilon (u+v+1)^{\alpha}.
\quad \left(u, v\in\mathbb{R}_{++}\right)
\]
This inequality, together with (\ref{Eq2.11}) and the triangle inequality
imply that
\begin{equation}\label{Eq2.12}
\left|F(u, v)-F(v, u)-F(u+1, v)-F(u, 1)+F(v+1, u)+F(v, 1)\right|\leq 2\varepsilon (u+v+1)^{\alpha}
\end{equation}
holds for all $u, v\in\mathbb{R}_{++}$.
On the other hand, with $u=1$, we get from (\ref{Eq2.7}) that
\[
\left|F(1+v, w)+F(1, v)-F(1+w, w)-F(1, w)\right|\leq \varepsilon(1+v+w)^{\alpha}.
\]
Replacing here $v$ by $u$ and $w$ by $v$, respectively, we have that
\[
\left|F(u+1, v)+F(1, u)-F(v+1, u)-F(1, v)\right|\leq \varepsilon (u+v+1)^{\alpha}.
\quad \left(u, v\in\mathbb{R}_{++}\right)
\]
Again, by the triangle inequality and the definitions
(\ref{Eq2.8}) and (\ref{Eq2.9}), (\ref{Eq2.12}) and the last inequality imply (\ref{Eq2.10}).

In what follows we will investigate the function $g$.
At this point of the proof we have to distinguish three cases.

\underline{\textsc{Case I.}} ($\alpha<0$)

In this case we will determine the function $g$ by proving that
\begin{equation}\label{Eq2.13}
g(u)=c(u^{\alpha}-1) \quad \left(u\in\mathbb{R}_{++}\right)
\end{equation}
with some $c\in \mathbb{R}$.

Indeed, (\ref{Eq2.10}), (\ref{Eq2.9}) and (\ref{Eq2.5}) imply that
\[
\left|G(tu, tv)-G(tv, tu)\right|\leq 3\varepsilon (tu+tv+1)^{\alpha},
\quad \left(t, u, v\in\mathbb{R}_{++}\right)
\]
therefore
\[
\left|t^{\alpha}F(u, v)+g(tv)-t^{\alpha}F(v, u)-g(tu)\right| \leq
3\varepsilon (tu+tv +1)^{\alpha}
\quad \left(t, u, v\in\mathbb{R}_{++}\right)
\]
or
\[
\left|F(u, v)-F(v, u)-t^{\alpha}\left(g(tu)-g(tv)\right)\right|
\leq 3\varepsilon (u+v+t^{-1})^{\alpha}
\quad \left(t, u, v\in\mathbb{R}_{++}\right)
\]
whence
\[
\lim_{t\rightarrow 0}t^{-\alpha}\left(g(tu)-g(tv)\right)=F(u, v)-F(v, u)
\quad \left(t, u, v\in \mathbb{R}_{++}\right)
\]
follows.
Particularly, with $v=1$, by (\ref{Eq2.8}), we have that
\begin{equation}\label{Eq2.14}
g(u)=\lim_{t\rightarrow 0}t^{-\alpha}\left(g(tu)-g(t)\right).
\quad \left(u\in\mathbb{R}_{++}\right)
\end{equation}
Let now $u, v\in\mathbb{R}_{++}$. Then, by (\ref{Eq2.14}), we obtain that
\[
\begin{array}{rcl}
g(uv)&=& \lim_{t\rightarrow 0}t^{-\alpha}\left[g(tuv)-g(t)\right] \\
 &=&\lim_{t\rightarrow 0}\left[(tv)^{-\alpha}\left(g((tv)u)-g(tv)\right)v^{\alpha}+
 t^{-\alpha}(g(tv)-g(t))\right] \\
  &=&g(u)v^{\alpha}+g(v).
\end{array}
\]
Therefore, $g(u)v^{\alpha}+g(v)=g(v)u^{\alpha}+g(u)$, that is,
\[
g(u)\left(v^{\alpha}-1\right)=g(v)\left(u^{\alpha}-1\right)
\quad \left(u, v\in\mathbb{R}_{++}\right)
\]
which implies (\ref{Eq2.13}) with $c=g(2)\left(2^{\alpha}-1\right)^{-1}$.

Thus, by (\ref{Eq2.6}), (\ref{Eq2.13}), (\ref{Eq2.9}) and (\ref{Eq2.10}), we have that
\begin{equation}\label{Eq2.15}
\begin{array}{l}
\left|f(x)-c(1-x)^{\alpha}-\left(f(1-x)-cx^{\alpha}\right)\right| \\
=\left|F(1-x, x)+cx^{\alpha}-\left(F(x, 1-x)+c(1-x)^{\alpha}\right)\right| \\
=\left|G(1-x, x)-G(x, 1-x)\right|\leq 3\cdot 2^{\alpha}\varepsilon
\end{array}
\end{equation}
holds for all $x\in ]0, 1[$.

In the next step we define the functions $f_{0}$ and
$F_{0}$ on $]0, 1[$ and on $]0, 1[^{2}$ by
\begin{equation}\label{Eq2.16}
f_{0}(x)=f(x)-c\left[(1-x)^{\alpha}-1\right]
\end{equation}
and
\begin{equation}\label{Eq2.17}
F_{0}(p, q)=f_{0}(p)+p^{\alpha}f_{0}(q)-f_{0}(pq)-
(1-pq)^{\alpha}f_{0}\left(\frac{1-p}{1-pq}\right),
\end{equation}
respectively.
Then (\ref{Eq2.1}) and (\ref{Eq2.15}) imply that
\begin{equation}\label{Eq2.18}
\left|f_{0}(x)+(1-x)^{\alpha}f_{0}\left(\frac{y}{1-x}\right)
-f_{0}(y)-(1-y)^{\alpha}f_{0}\left(\frac{x}{1-y}\right)\right|\leq \varepsilon
\end{equation}
for all $(x, y)\in D^{\circ}$ and
\begin{equation}\label{Eq2.19}
\left|f_{0}(x)-f_{0}(1-x)\right|\leq 3\cdot 2^{\alpha}\varepsilon.
\quad \left(x\in ]0, 1[\right)
\end{equation}
Furthermore, with the substitutions $x=1-p$, $y=pq$ ($p, q\in ]0, 1[$), (\ref{Eq2.18})
implies that
\begin{equation}\label{Eq2.20}
\left|f_{0}(1-p)+p^{\alpha}f_{0}(q)-
f_{0}(pq)-(1-pq)^{\alpha}f_{0}\left(\frac{1-p}{1-pq}\right)\right|\leq\varepsilon
\end{equation}
holds for all $p, q\in ]0, 1[$. Therefore, due to (\ref{Eq2.19}) and the
triangle inequality, (\ref{Eq2.18}) implies that
\begin{equation}\label{Eq2.21}
\left|F_{0}(p, q)\right|\leq \left(3\cdot 2^{\alpha}+1\right)\varepsilon.
\quad \left(p, q\in ]0, 1[\right)
\end{equation}
It can easily be checked that
\begin{multline}\label{Eq*}
f_{0}(p)\left[q^{\alpha}+(1-q)^{\alpha}-1\right]-f_{0}(q)\left[p^{\alpha}+(1-p)^{\alpha}-1\right] \\
=F_{0}(q, p)-F_{0}(p, q)-(1-pq)^{\alpha}
\left[F_{0}\left(\frac{1-q}{1-pq}, p\right)+
f_{0}\left(1-\frac{1-p}{1-pq}\right)
-f_{0}\left(\frac{1-p}{1-pq}\right)\right]
\end{multline}
holds for all $p, q\in ]0, 1[$.
Thus, by (\ref{Eq2.21}) and (\ref{Eq2.19}) we get that

\begin{multline*}
\left|f_{0}(p)-\frac{f_{0}(q)}{q^{\alpha}+(1-q)^{\alpha}-1}
\left[p^{\alpha}+(1-p)^{\alpha}-1\right]\right| \\
\leq
\frac{2(1+3\cdot 2^{\alpha})+(1-pq)^{\alpha}(1+6\cdot 2^{\alpha})}{q^{\alpha}+(1-q)^{\alpha}-1}\varepsilon.
\quad \left(p, q\in ]0, 1[\right)
\end{multline*}
Taking into consideration (\ref{Eq2.16}), with $q=\frac{1}{2}$ with the definitions
$a=f_{0}\left(\frac{1}{2}\right)\left(2^{1-\alpha}-1\right)^{-1}$, $b=a+c$,
this inequality implies that
\[
\left|f(x)-\left[ax^{\alpha}+b(1-x)^{\alpha}-b\right]\right| \leq
\frac{8+6\cdot 2^{\alpha}+2^{-\alpha}}{2^{1-\alpha}-1}\varepsilon.
\quad  \left(x\in ]0, 1[\right)
\]
In view of the definition of
$K(\alpha)$, this implies that inequality (\ref{Eq2.3}) holds for all
$x\in ]0, 1[$.

\underline{\textsc{Case II.}} ($\alpha=0$)

In the second case we will show that there exists a logarithmic function
$l:\mathbb{R}_{++}\rightarrow\mathbb{R}$ such that
\[
\left|g(u)-l(u)\right|\leq 6\varepsilon
\]
for all $u\in\mathbb{R}_{++}$.
Indeed, (\ref{Eq2.10}) yields in this case that
\[
\left|G(u, v)-G(v, u)\right|\leq 3\varepsilon.
\quad \left(u, v\in \mathbb{R}_{++}\right)
\]
Due to (\ref{Eq2.5}) and (\ref{Eq2.9}) we obtain that
\[
\begin{array}{rcl}
G(tu, tv)&=&F(tu, tv)+g(tv)\\
&=&F(u, v)+g(tv)\\
&=&G(u, v)-g(v)+g(tv)
\end{array}
\]
that is,
\[
G(tu, tv)-G(u, v)=g(tv)-g(v),  \quad \left(t, u, v\in\mathbb{R}_{++}\right)
\]
therefore
\begin{equation}\label{Eq2.22}
\begin{array}{l}
\left|g(tv)-g(v)+g(u)-g(tu)\right| \\
=\left|G(tu, tv)-G(u, v)-G(tv, tu)+G(v, u)\right| \\
\leq \left|G(tu, tv)-G(tv, tu)\right|+\left|G(v, u)-G(u, v)\right|
\leq 6\varepsilon
\end{array}
\end{equation}
for all $t, u, v\in\mathbb{R}_{++}$. Now (\ref{Eq2.22}) with the substitution
$u=1$ implies that
\[
\left|g(tv)-g(v)-g(t)\right|\leq 6\varepsilon
\]
holds for all $t, v\in\mathbb{R}_{++}$, since obviously $g(1)=0$.
This means that the function $g$ is approximately logarithmic on
$\mathbb{R}_{++}$. Thus (see e.g. Forti \cite{For95}) there exists
a logarithmic function $l:\mathbb{R}_{++}\rightarrow\mathbb{R}$ such that
\[
\left|g(u)-l(u)\right|\leq 6\varepsilon
\]
holds for all $u\in \mathbb{R}_{++}$.

Furthermore,
\begin{equation}\label{Eq2.23}
\begin{array}{l}
\left|f(x)-l(1-x)-\left(f(1-x)-l(x)\right)\right| \\
=\left|F(1-x, x)-l(1-x)-F(x, 1-x)+l(x)\right| \\
=\left|F(1-x, x)+g(x)-g(x)-l(1-x)\right. \\
-F(x, 1-x)+g(1-x)-g(1-x)+l(x)\left.\right| \\
\leq \left|F(1-x, x)+g(x)-\left(F(x, 1-x)+g(1-x)\right)\right| \\
+\left|g(1-x)-l(1-x)\right|+\left|l(x)-g(x)\right| \\
=\left|G(1-x, x)-G(x, 1-x)\right|\\
+\left|g(1-x)-l(1-x)\right|+\left|l(x)-g(x)\right| \\
\leq 3\varepsilon+6\varepsilon+6\varepsilon=15 \varepsilon
\end{array}
\end{equation}
As in the first part of the proof, define the functions $f_{0}$ and
$F_{0}$ on $]0, 1[$ and on $]0, 1[^{2}$, respectively, by
\[
f_{0}(x)=f(x)-l(1-x)
\]
and
\[
F_{0}(p, q)=f_{0}(p)+f_{0}(q)-f_{0}(pq)-f_{0}\left(\frac{1-p}{1-pq}\right)
\]
Due to (\ref{Eq2.23})
\begin{equation}\label{Eq2.24}
\left|f_{0}(x)-f_{0}(1-x)\right|\leq 15\varepsilon
\end{equation}
holds for all $x\in ]0, 1[$.
Furthermore, with the substitutions $x=1-p$, $y=pq$ ($p, q\in ]0, 1[$)
inequality (\ref{Eq2.1}) implies, that
\begin{equation}\label{Eq2.25}
\left|f_{0}(1-p)+f_{0}(q)-
f_{0}(pq)-f_{0}\left(\frac{1-p}{1-pq}\right)\right|\leq\varepsilon
\end{equation}
is fulfilled for all $p, q\in ]0, 1[$.
Inequalities (\ref{Eq2.24})  and (\ref{Eq2.25}) and the triangle inequality
imply that
\begin{equation}\label{Eq2.26}
\left|F_{0}(p, q)\right|\leq 16\varepsilon
\end{equation}
for all $p, q\in ]0, 1[$.
An easy calculation shows that
\begin{multline*}
f_{0}(p)-f_{0}(q)\\
=F_{0}(q, p)-F_{0}(p, q)+F_{0}\left(\frac{1-p}{1-pq}, p\right)-
f_{0}\left(1-\frac{1-p}{1-pq}\right)+f_{0}\left(\frac{1-p}{1-pq}\right)
\end{multline*}
therefore,
\begin{equation}\label{Eq2.27}
\begin{array}{l}
\left|f_{0}(p)-f_{0}(q)\right|\\
\leq
\left|F_{0}(q, p)\right|+\left|F_{0}(p, q)\right|+
\left|F_{0}\left(\frac{1-p}{1-pq}, p\right)\right|+
\left|f_{0}\left(1-\frac{1-p}{1-pq}\right)-f_{0}\left(\frac{1-p}{1-pq}\right)\right| \\
\leq 3\cdot 16\varepsilon+15\varepsilon=63\varepsilon
\end{array}
\end{equation}
holds for all $p, q\in ]0, 1[$.
With the substitution $q=\frac{1}{2}$ inequality (\ref{Eq2.27}) implies that
\[
\left|f_{0}(p)-f_{0}\left(\frac{1}{2}\right)\right| \leq 63\varepsilon.
\quad \left(p\in ]0, 1[\right)
\]
Using the definition of the function $f_{0}$, we obtain that inequality
\[
\left|f(x)-l(1-x)-c\right|\leq 63\varepsilon
\]
is satisfied for all $x\in]0, 1[$, where $c=f_{0}\left(\frac{1}{2}\right)$.
Hence inequality (\ref{Eq2.2}) holds, indeed.

\underline{\textsc{Case III.}} ($1\neq\alpha>0$)

Finally, in the last case, we will prove that
there exists $c\in\mathbb{R}$ such that
\[
\left|g(x)-c(x^{\alpha}-1)\right|
\leq \frac{4\cdot 3^{\alpha+1}\varepsilon}{\left|2^{-\alpha}-1\right|}
\]
holds for all $x\in ]0, 1[$.

Due to inequalities (\ref{Eq2.4}) and (\ref{Eq2.8}),
\[
\begin{array}{lcl}
G(tu, tv)&=&F(tu, tv)+g(tv)\\
 &=&t^{\alpha}F(u, v)+g(tv)\\
 &=&t^{\alpha}G(u, v)-t^{\alpha}g(v)+g(tv),
\end{array}
\]
that is,
\[
G(tu, tv)-t^{\alpha}G(u, v)=g(tv)-t^{\alpha}g(v)
\]
holds for all $t, v\in\mathbb{R}_{++}$.
Therefore,
\begin{equation}\label{Eq2.28}
\begin{array}{l}
\left|g(tv)-t^{\alpha}g(v)+t^{\alpha}g(u)-g(tu)\right| \\
=\left|G(tu, tv)-G(u, v)-G(tv, tu)+G(v, u)\right| \\
\leq \left|G(tu, tv)-G(tv, tu)\right|+\left|G(u, v)-G(v, u)\right| \\
\leq 3\varepsilon(t(u+v)+1)^{\alpha}+3\varepsilon(u+v+1)^{\alpha}
\end{array}
\end{equation}
holds for all $t, u, v\in\mathbb{R}_{++}$, where we used (\ref{Eq2.10}).
With the substitution $u=1$, (\ref{Eq2.28}) implies that
\begin{multline}\label{Eq2.29}
\left|g(tv)-t^{\alpha}g(v)-g(t)\right| \\
\leq 3\varepsilon(t(v+1)+1)^{\alpha}+3\varepsilon(v+2)^{\alpha}
\quad \left(t, v\in\mathbb{R}_{++}\right)
\end{multline}
Interchanging $t$ and $v$ in (\ref{Eq2.29}), we obtain that
\begin{multline}\label{Eq2.30}
\left|g(tv)-v^{\alpha}g(t)-g(v)\right| \\
\leq 3\varepsilon(v(t+1)+1)^{\alpha}+3\varepsilon(t+2)^{\alpha}
\quad \left(t, v\in\mathbb{R}_{++}\right)
\end{multline}
Inequalities (\ref{Eq2.29}), (\ref{Eq2.30}) and the triangle inequality
imply that
\begin{equation}\label{Eq2.31}
\left|t^{\alpha}g(v)+g(t)-v^{\alpha}g(t)-g(v)\right|
\leq B(t, v)
\end{equation}
is fulfilled for all $t, v\in\mathbb{R}_{++}$, where
\[
\begin{array}{lcl}
B(t, v)&=& 3\varepsilon(t(v+1)+1)^{\alpha}+3\varepsilon(v+2)^{\alpha} \\
 &+&3\varepsilon(v(t+1)+1)^{\alpha}+3\varepsilon(t+2)^{\alpha}.
\end{array}
\]
With the substitution $t=\frac{1}{2}$ and with the definition
$c=\frac{g\left(\frac{1}{2}\right)}{2^{-\alpha}-1}$, we obtain
\begin{equation}\label{Eq2.32}
\left|g(v)-c(v^{\alpha}-1)\right|\leq \frac{B\left(\frac{1}{2}, v\right)}{\left|2^{-\alpha}-1\right|}
\end{equation}
for all $v\in\mathbb{R}_{++}$.

Let us observe that
\[
\left|B(t, v)\right|\leq 4\cdot 3^{\alpha+1}\varepsilon
\]
holds, if $t, v\in ]0, 1[$.
Thus
\begin{equation}\label{Eq2.33}
\left|g(v)-c(v^{\alpha}-1)\right|\leq \frac{B\left(\frac{1}{2}, v\right)}{\left|2^{-\alpha}-1\right|}
\leq \frac{4\cdot 3^{\alpha+1}\varepsilon}{\left|2^{-\alpha}-1\right|}
\end{equation}
for all $v\in ]0, 1[$.
Therefore (\ref{Eq2.6}), (\ref{Eq2.9}), (\ref{Eq2.10}), (\ref{Eq2.33}) and the
triangle inequality imply that
\begin{equation}\label{Eq2.34}
\begin{array}{l}
\left|f(x)-c(1-x)^{\alpha}+c-\left(f(1-x)-cx^{\alpha}+c\right)\right| \\
=\left|F(1-x, x)-c(1-x)^{\alpha}+c-\left(F(x, 1-x)-cx^{\alpha}+c\right)\right| \\
\leq \left|F(1-x, x)+g(x)-F(x, 1-x)-g(1-x)\right| \\
+ \left|g(x)-c(x^{\alpha}-1)\right|+\left|g(1-x)-c((1-x)^{\alpha}-1)\right| \\
=\left|G(1-x, x)-G(x, 1-x)\right| \\
+ \left|g(x)-c(x^{\alpha}-1)\right|+\left|g(1-x)-c((1-x)^{\alpha}-1)\right| \\
\leq 3\cdot 2^{\alpha}\varepsilon+\frac{8\cdot 3^{\alpha+1}\varepsilon}{\left|2^{-\alpha}-1\right|}
\end{array}
\end{equation}
holds for all $x\in ]0, 1[$.

As in the previous cases, we define the functions $f_{0}$ and
$F_{0}$ on $]0, 1[$ and on $]0, 1[^{2}$ by
\begin{equation}\label{Eq2.35}
f_{0}(x)=f(x)-c(1-x)^{\alpha}
\end{equation}
and
\begin{equation}\label{Eq2.36}
F_{0}(p, q)=f_{0}(p)+p^{\alpha}f_{0}(q)-f_{0}(pq)-
(1-pq)^{\alpha}f_{0}\left(\frac{1-p}{1-pq}\right),
\end{equation}
respectively.
Then (\ref{Eq2.1}), (\ref{Eq2.34}) and (\ref{Eq2.35}) imply that
\begin{equation}\label{Eq2.37}
\left|f_{0}(x)+(1-x)^{\alpha}f_{0}\left(\frac{y}{1-x}\right)
-f_{0}(y)-(1-y)^{\alpha}f_{0}\left(\frac{x}{1-y}\right)\right|\leq \varepsilon
\end{equation}
for all $(x, y)\in D^{\circ}$ and
\begin{equation}\label{Eq2.38}
\left|f_{0}(x)-f_{0}(1-x)\right|\leq 3\cdot 2^{\alpha}\varepsilon+\frac{8\cdot 3^{\alpha+1}\varepsilon}{\left|2^{-\alpha}-1\right|}.
\quad \left(x\in ]0, 1[\right)
\end{equation}
Furthermore, with the substitutions $x=1-p$, $y=pq$ ($p, q\in ]0, 1[$), (\ref{Eq2.37})
implies that
\begin{equation}\label{Eq2.39}
\left|f_{0}(1-p)+p^{\alpha}f_{0}(q)-
f_{0}(pq)-(1-pq)^{\alpha}f_{0}\left(\frac{1-p}{1-pq}\right)\right|\leq\varepsilon
\end{equation}
holds for all $p, q\in ]0, 1[$.
Thus (\ref{Eq2.38}) and (\ref{Eq2.39}) and the triangle
inequality imply that
\[
\left|F_{0}(p, q)\right|\leq \varepsilon+3\cdot 2^{\alpha}\varepsilon+\frac{8\cdot 3^{\alpha+1}\varepsilon}{\left|2^{-\alpha}-1\right|} .
\quad \left(x\in ]0, 1[\right)
\]
As in the previous cases, it is easy to see that the identity (\ref{Eq*})
is satisfied for all $p, q\in ]0, 1[$. Therefore
\begin{multline*}
\left|f_{0}(p)-\frac{f_{0}(q)}{q^{\alpha}+(1-q)^{\alpha}-1}
\left[p^{\alpha}+(1-p)^{\alpha}-1\right]\right| \\
\leq \left|q^{\alpha}+(1-q)^{\alpha}-1\right|^{-1}\left(3\left(\varepsilon+3\cdot 2^{\alpha}\varepsilon+\frac{8\cdot 3^{\alpha+1}\varepsilon}{\left|2^{-\alpha}-1\right|}\right)+
3\cdot 2^{\alpha}\varepsilon+\frac{8\cdot 3^{\alpha+1}\varepsilon}{\left|2^{-\alpha}-1\right|}\right)
\end{multline*}

for all $p, q\in ]0, 1[$.
In view of (\ref{Eq2.35}), with $q=\frac{1}{2}$ with the definitions
\[
a=f_{0}\left(\frac{1}{2}\right)\left(2^{1-\alpha}-1\right)^{-1}\quad  \text{and } \quad
b=a+c,
\]
this inequality implies that
\begin{equation}
\left|f(p)-\left[ap^{\alpha}+b(1-p)^{\alpha}-b\right]\right|\leq K(\alpha)\varepsilon
\end{equation}
holds for all $p\in ]0, 1[$, where
\[
K(\alpha)=
\left|2^{1-\alpha}-1\right|^{-1}
\left(3+12\cdot2^{\alpha}+\frac{32\cdot3^{\alpha+1}}{\left|2^{-\alpha}-1\right|}\right),
\]
which had to be proved.
\end{proof}

\section{Corollaries and remarks}

In the last part of the paper, first we explain, why our method does not work, in case $\alpha=1$.
\begin{Rem}
Since
\[
\lim_{\alpha\rightarrow 1}K(\alpha)=+\infty,
\]
our method is inappropriate if $\alpha=1$. Hence we cannot prove stability
concerning the fundamental equation of information on the set $D^{\circ}$.

An easy calculation shows that
\[
\sup_{\alpha<0}K(\alpha)=\sup_{\alpha<0}\frac{8+6\cdot 2^{\alpha}+2^{-\alpha}}{2^{1-\alpha}-1}=15,
\]
therefore, in case $\alpha<0$, $15$ can also be considered as a stability constant.
\end{Rem}

Using Theorem \ref{Thm2.1}., with the choice $\varepsilon=0$,
we get the general solution of equation (\ref{Eq1.1})
(see Ebanks--Sahoo--Sander \cite{ESS98} or Maksa \cite{Mak82}).
\begin{Cor}\label{C1}
Let $\alpha\neq 1$ be arbitrary but fixed real number and assume
that the function $f:]0, 1[\rightarrow\mathbb{R}$ satisfies the
functional equation
\[
f(x)+(1-x)^{\alpha}f\left(\frac{y}{1-x}\right)=
f(y)+(1-x)^{\alpha}f\left(\frac{x}{1-y}\right)
\]
for all pairs $(x, y)\in D^{\circ}$. Then, and only then,
in case $\alpha=0$, there exists a logarithmic function
$l:]0, 1[\rightarrow\mathbb{R}$ and $c\in\mathbb{R}$ such that
\[
f(x)=l(1-x)+c,
\quad \left(x\in ]0, 1[\right)
\]
furthermore, in case $\alpha\notin \left\{0, 1\right\}$, there exist
$a, b\in\mathbb{R}$ such that
\[
f(x)=ax^{\alpha}+b(1-x)^{\alpha}-b
\]
holds for all $x\in ]0, 1[$.
\end{Cor}

\begin{Rem}
In view of Corollary \ref{C1}., Theorem \ref{Thm2.1}. says that the
parametric fundamental equation of information is stable on the
open domain in the sense of Hyers and Ulam, provided that the
parameter does not equal to one.
\end{Rem}

\begin{Rem}
Let us observe that the solutions of (\ref{Eq1.1}) are bounded on
$D^{\circ}$, assuming that $1 \neq \alpha>0$. Therefore
Theorem \ref{Thm2.1}. means that the parametric
fundamental equation of information is not only stable but also superstable
in this case(as to the superstability,
the reader can consult Ger \cite{Ger94} and Moszner \cite{Mos04}).
\end{Rem}

In the following theorem we shall prove that equation
(\ref{Eq1.1}) is stable not only on $D^{\circ}$ but also
on $D$. During the proof of this theorem the following function
will be needed.
For all $1\neq \alpha >0$ we define the function $T(\alpha)$ by
\[
T(\alpha)= 3\cdot 2^{\alpha} +\frac{8\cdot 3^{\alpha+1}}{\left|2^{-\alpha}-1\right|},
\]
that is, $T(\alpha)$ is that function which appears in inequality \eqref{Eq2.38}.
Furthermore, the following relationship is fulfilled between $K(\alpha)$ and $T(\alpha)$
\[
K(\alpha)=\frac{4T(\alpha)+3}{\left|2^{1-\alpha}-1\right|}
\]
for all $1\neq \alpha>0$.
\begin{Thm}\label{Thm3.1}
Let $\alpha, \varepsilon\in\mathbb{R}$ be fixed, $\alpha\neq 1$, $\varepsilon\geq 0$.
Suppose that the function $f:[0,1]\rightarrow\mathbb{R}$ satisfies inequality
(\ref{Eq2.1}) for all $(x, y)\in D$.
Then, in case $\alpha\neq 0$ there
exist $a, b\in\mathbb{R}$ such that the function $h_{1}$
defined on $[0, 1]$ by
\[
h_{1}(x)=\left\{
\begin{array}{lcl}
0, & \text{if} & x=0\\
ax^{\alpha}+b(1-x)^{\alpha}-b, &\text{if}& x\in \left.]0, 1[\right. \\
a-b, & \text{if} & x=1
\end{array}
\right.
\]
is a solution of (\ref{Eq1.1}) on $D$ and
\begin{equation}\label{Eq3.1}
\left|f(x)-h_{1}(x)\right|\leq K(\alpha)\varepsilon,
\quad \left(x\in [0, 1]\right)
\end{equation}
holds if $\alpha<0$ and
\begin{equation}\label{Eq4.2}
\left|f(x)-h_{1}(x)\right|\leq \max\left\{K(\alpha), T(\alpha)+1\right\}\varepsilon,
\quad \left(x\in [0, 1]\right)
\end{equation}
is satisfied in case $1\neq\alpha>0$.
Furthermore, in case $\alpha=0$, there exists $c\in\mathbb{R}$ such
that the function $h_{2}$ defined on $[0, 1]$ by
\[
h_{2}(x)=\left\{
\begin{array}{lcl}
f(0), & \text{if} & x=0\\
c, &\text{if}& x\in \left.]0, 1[\right. \\
f(1), & \text{if} & x=1
\end{array}
\right.
\]
is a solution of (\ref{Eq1.1}) on $D$ and
\begin{equation}\label{Eq3.2}
\left|f(x)-h_{2}(x)\right|\leq K(\alpha)\varepsilon.
\quad \left(x\in [0, 1]\right)
\end{equation}
\end{Thm}

\begin{proof}
An easy calculation shows that the functions
$h_{1}$ and $h_{2}$ are the solutions of equation (\ref{Eq1.1})
on $D$ in case $\alpha\neq 0$ and $\alpha=0$, respectively.

Firstly, we investigate the case $\alpha<0$.
Theorem \ref{Thm2.1}. implies that (\ref{Eq3.1}) holds for all
$x\in ]0, 1[$.
Thus it is enough to prove that (\ref{Eq3.1}) holds for $x=0$ and $x=1$.
It follows from (\ref{Eq2.1}), with the substitution $y=0$, that
\[
\left|(1-x)^{\alpha}-1\right|\cdot \left|f(0)\right|\leq \varepsilon.
\quad \left(x\in ]0, 1[\right)
\]
Since $\alpha<0$, $f(0)=0$ follows, that is, (\ref{Eq3.1})
holds in case $x=0$.

Let now $x\in ]0, 1[$ and $y=1-x$ in (\ref{Eq2.1}).
Then
\[
\left|f(1-x)-f(x)-f(1)\left((1-x)^{\alpha}-x^{\alpha}\right)\right|\leq \varepsilon.
\]
Applying (\ref{Eq2.3}) to $1-x$ instead of $x$ to get
\[
\left|-f(1-x)+a(1-x)^{\alpha}+bx^{\alpha}-b\right|\leq K(\alpha)\varepsilon.
\]
Adding this last two inequalities and inequality (\ref{Eq2.3}) up
and using the triangle inequality  to obtain
\[
\left|(a-b)-f(1)\right|\cdot \left|(1-x)^{\alpha}-x^{\alpha}\right|\leq \left(2K(\alpha)+1\right)\varepsilon.
\quad \left(x\in ]0, 1[\right)
\]
Since $\alpha<0$, we get that $f(1)=a-b$ and so (\ref{Eq3.1}) holds also for $x=1$.

Secondly, we deal with the case $\alpha>0$.
Substituting $x=0$ into (\ref{Eq2.1}) and with $y\rightarrow 0$ we obtain that
\[
\left|f(0)\right|\leq \varepsilon \leq K(\alpha) \varepsilon,
\]
that is, (\ref{Eq3.1}) holds for $x=0$.
If $x\in ]0, 1[$, then inequality \eqref{Eq4.2} follows immediately from
Theorem \ref{Thm2.1}.
Furthermore, with the substitution $y=1-x$ ($x\in ]0, 1[$) inequality \eqref{Eq2.1} implies that
\[
\left|f(x)+(1-x)^{\alpha}f(1)-f(1-x)-x^{\alpha}f(1)\right|\leq \varepsilon .
\quad \left(x\in ]0, 1[\right)
\]
From the proof of Theorem \ref{Thm2.1} (see definition \eqref{Eq2.35}) it is known that
\[
f(x)=f_{0}(x)+c(1-x)^{\alpha} ,
\quad \left(x\in ]0, 1[\right)
\]
therefore the last inequality yields that
\begin{equation}\label{Eq4.3}
\left|f_{0}(x)-f_{0}(1-x)+c(1-x)^{\alpha}-cx^{\alpha}+(1-x)^{\alpha}f(1)-x^{\alpha}f(1)\right|\leq \varepsilon
\end{equation}
holds for all $x\in ]0, 1[$.
Whereas
\[
\left|f_{0}(x)-f_{0}(1-x)\right|\leq T(\alpha).
\quad \left(x\in ]0, 1[\right)
\]
Thus after rearranging \eqref{Eq4.3}, we get that
\[
\left|f_{0}(x)-f_{0}(1-x)-[c+f(1)][x^{\alpha}-(1-x)^{\alpha}]\right|\leq \varepsilon,
\quad \left(x\in ]0, 1[\right)
\]
that is,
\[
\left|\left|f_{0}(x)-f_{0}(1-x)\right|-\left|c+f(1)\right|\cdot \left|x^{\alpha}-(1-x)^{\alpha}\right|\right|
\leq \varepsilon
\]
holds for all $x\in ]0, 1[$. Therefore
\[
\left|c+f(1)\right|\cdot \left|x^{\alpha}-(1-x)^{\alpha}\right|\leq (T(\alpha)+1)\varepsilon
\]
for all $x\in ]0, 1[$.
Taking the limit $x\rightarrow 0+$, we obtain that
\[
\left|c+f(1)\right|\leq (T(\alpha)+1)\varepsilon.
\]
However, in the proof of Theorem \ref{Thm2.1}. we used the definition $c=b-a$, thus
\[
\left|f(1)-(a-b)\right|\leq (T(\alpha)+1)\varepsilon,
\]
so \eqref{Eq4.2} holds, indeed.

Finally, we investigate the case $\alpha=0$.
If $x=0$ or $x=1$, then \eqref{Eq3.2} trivially holds, since
\[
\left|f(0)-h_{2}(0)\right|=\left|f(0)-f(0)\right|=0\leq K(\alpha)\varepsilon
\]
and
\[
\left|f(1)-h_{2}(1)\right|=\left|f(1)-f(1)\right|=0\leq K(\alpha)\varepsilon .
\]
Let now $x\in ]0, 1[$ and $y=1-x$ in (\ref{Eq2.1}), then we obtain that
\begin{equation}\label{Eq3.3}
\left|f(x)-f(1-x)\right|\leq \varepsilon,  \quad \left(x\in ]0, 1[\right)
\end{equation}
if fulfilled for all $x\in ]0, 1[$.

Due to Theorem \ref{Thm2.1}. there
exists a logarithmic function $l:]0, 1[\rightarrow\mathbb{R}$ and
$c\in\mathbb{R}$ such that
\begin{equation}\label{Eq3.4}
\left|f(x)-l(1-x)-c\right|\leq 63\varepsilon
\end{equation}
holds for all $x\in ]0, 1[$. Hence it is enough to prove that the
function $l$ is identically zero on $]0, 1[$.
Indeed, due to (\ref{Eq2.2}) and (\ref{Eq3.3})
\begin{multline}\label{Eq3.5}
\left|l(1-x)-l(x)\right| \\
=\left|l(1-x)-f(1-x)+f(1-x)+c
-l(x)+f(x)-f(x)-c\right| \\
\leq
\left|l(1-x)+c-f(x)\right|+\left|f(1-x)-l(x)-c\right|
+\left|f(x)-f(1-x)\right| \\
\leq 127 \varepsilon
\end{multline}
holds for all $x\in ]0, 1[$. Since the function $l$
is uniquely extendable to $\mathbb{R}_{++}$, with the
substitution $x=\frac{p}{p+q}$ ($p, q\in\mathbb{R}$), we get that
\[
\left|l(p)-l(q)\right|\leq 127\varepsilon,  \quad \left(p, q\in \mathbb{R}_{++}\right)
\]
where we used the fact that $l$ is logarithmic, as well.
This last inequality, with the substitution $q=1$ implies that
\[
\left|l(p)\right|\leq 127\varepsilon
\]
holds for all $p\in\mathbb{R}_{++}$, since $l(1)=0$. Thus
$l$ is bounded on $\mathbb{R}_{++}$.
However, the only bounded, logarithmic function on
$\mathbb{R}_{++}$ is the identically zero function.
Therefore,
\[
\left|f(x)-c\right|\leq 63\varepsilon
\]
holds for all $x\in ]0, 1[$, i.e., (\ref{Eq3.2}) is proved.
\end{proof}

Applying Theorem \ref{Thm2.1}. we can prove the stability of a
system of functional equations that characterizes the
$\alpha$-recursive, $3$-semi-symmetric information measures.

\begin{Thm}\label{Thm3.2}
Let $n\geq 2$ be a fixed positive integer
and $(I_{n})$ be the sequence of functions
$I_{n}:\Gamma^{\circ}_{n}\rightarrow\mathbb{R}$
and suppose that there exist a sequence $(\varepsilon_{n})$
of nonnegative real numbers and a real number $\alpha\neq 1$
such that
\begin{multline}\label{Eq3.6}
\left|I_{n}(p_{1}, \ldots, p_{n})\right.\\
-I_{n-1}(p_{1}+p_{2}, p_{3}, \ldots, p_{n})-
\left.(p_{1}+p_{2})^{\alpha}I_{2}\left(\frac{p_{1}}{p_{1}+p_{2}}, \frac{p_{2}}{p_{1}+p_{2}}\right)\right|
\leq
\varepsilon_{n-1}
\end{multline}
for all $n\geq 3$ and $(p_{1}, \ldots, p_{n})\in\Gamma^{\circ}_{n}$, and
\begin{equation}\label{Eq3.7}
\left|I_{3}(p_{1}, p_{2}, p_{3})-I_{3}(p_{1}, p_{3}, p_{2})\right|\leq \varepsilon_{1}
\end{equation}
holds on $\Gamma^{\circ}_{n}$.
Then, in case $\alpha<0$ there exist $c, d\in\mathbb{R}$ such that
\begin{multline}\label{Eq3.8}
\left|I_{n}\left(p_{1}, \ldots, p_{n}\right)-
\left[c H^{\alpha}_{n}\left(p_{1}, \ldots, p_{n}\right)+d\left(p_{1}^{\alpha}-1\right)\right]\right| \\
\leq \sum^{n-1}_{k=2}\varepsilon_{k}+K(\alpha)\left(2\varepsilon_{2}+\varepsilon_{1}\right)
\left(1+\sum^{n-1}_{k=2}\left(\sum^{k}_{i=1}p_{i}^{\alpha}\right)\right)
\end{multline}
for all $n\geq2$ and $\left(p_{1}, \ldots, p_{n}\right)\in\Gamma^{\circ}_{n}$.
Furthermore, in case $\alpha=0$ there exists a logarithmic function
$l:]0, 1[\rightarrow\mathbb{R}$ and $c\in\mathbb{R}$ such that
\begin{multline}\label{Eq3.9}
\left|I_{n}\left(p_{1}, \ldots, p_{n}\right)-\left[cH^{0}_{n}\left(p_{1}, \ldots, p_{n}\right)+l(p_{1})\right]\right|\\
\leq \sum^{n-1}_{k=2}\varepsilon_{k}+K(\alpha)\left(n-1\right)\left(2\varepsilon_{2}+\varepsilon_{1}\right)
\end{multline}
for all $n\geq 2$ and $\left(p_{1}, \ldots, p_{n}\right)\in\Gamma^{\circ}_{n}$.
Finally, if $\alpha>0$ then
there exist $c, d\in\mathbb{R}$ such that
\begin{multline}\label{Eq3.10}
\left|I_{n}(p_{1}, \ldots, p_{n})-
\left[cH^{\alpha}_{n}(p_{1}, \ldots, p_{n})+d(p^{\alpha}_{1}-1)\right]\right|\\
\leq \sum^{n-1}_{k=2}\varepsilon_{k}+(n-1)K(\alpha)(2\varepsilon_{2}+\varepsilon_{1})
\end{multline}
holds for all $n\geq 2$ and $(p_{1}, \ldots, p_{n})\in\Gamma^{\circ}_{n}$,
where the convention
\[
\sum^{1}_{k=2}\varepsilon_{k}=\sum^{1}_{k=2}\left(\sum^{k}_{i=1}p_{i}^{\alpha}\right)=0
\]
is adopted.
\end{Thm}

\begin{proof}
As in \cite{Mak08}, due to (\ref{Eq3.6}) and (\ref{Eq3.7}),
it can be proved that,
for the function $f$ defined on $]0, 1[$ by
$f(x)=I_{2}(1-x, x)$ we get that
\[
\left|f(x)+(1-x)^{\alpha}f\left(\frac{y}{1-x}\right)
-f(y)-(1-y)^{\alpha}f\left(\frac{x}{1-y}\right)\right|\leq 2\varepsilon_{2}+\varepsilon_{1}
\]
for all $(x, y)\in D^{\circ}$, i.e., (\ref{Eq2.1}) holds with
$\varepsilon=2\varepsilon_{2}+\varepsilon_{1}$.
Therefore, applying Theorem \ref{Thm2.1}. we obtain
(\ref{Eq2.2}) and (\ref{Eq2.3}), respectively, with some
$a, b, c\in\mathbb{R}$ and a logarithmic function $l:]0, 1[\rightarrow\mathbb{R}$
and
$\varepsilon=2\varepsilon_{2}+\varepsilon_{1}$, i.e.,

\[
\left|I_{2}\left(1-x, x\right)-\left(ax^{\alpha}+b(1-x)^{\alpha}-b\right)\right|\leq K(\alpha)(2\varepsilon_{2}+\varepsilon_{1}),
\quad \left(x\in ]0, 1[\right)
\]
in case $\alpha\neq 0$, and
\[
\left|I_{2}\left(1-x, x\right)-\left(l(1-x)+c\right)\right|\leq K(\alpha)(2\varepsilon_{2}+\varepsilon_{1})
\quad \left(x\in ]0, 1[\right)
\]
in case $\alpha=0$.

Therefore (\ref{Eq3.8}) and (\ref{Eq3.10}) holds  with $c=(2^{1-\alpha}-1)a$,
$d=b-a$ in case $\alpha<0$ and in case $\alpha>0$,
furthermore, (\ref{Eq3.9}) holds in case $\alpha=0$, respectively, for $n=2$.

We continue the proof by induction on $n$.
Suppose that (\ref{Eq3.8}), (\ref{Eq3.9}) and (\ref{Eq3.10}) holds, resp., and for the
sake of brevity, introduce the notation
\[
J_{n}(p_{1}, \ldots, p_{n})=\left\{
\begin{array}{lcl}
cH^{\alpha}_{n}(p_{1}, \ldots, p_{n}),&\text{if}& \alpha\neq 0 \\
cH^{0}_{n}(p_{1}, \ldots, p_{n})+l(p_{1}),&\text{if}& \alpha=0
\end{array}
\right.
\]
for all $n\geq 2$, $(p_{1}, \ldots, p_{n})\in\Gamma^{\circ}_{n}$.
It can easily be seen that (\ref{Eq3.8}), (\ref{Eq3.9}) and (\ref{Eq3.10}) hold
on $\Gamma^{\circ}_{n}$ for $J_{n}$ instead of $I_{n}$ ($n\geq 3$)
with $\varepsilon_{n}=0$ ($n\geq 2$).
Thus, for all $(p_{1}, \ldots, p_{n+1})\in\Gamma^{\circ}_{n+1}$, we get that
\[
\begin{array}{l}
I_{n+1}(p_{1}, \ldots, p_{n+1})-J_{n+1}(p_{1}, \ldots, p_{n+1}) \\
=I_{n+1}(p_{1}, \ldots, p_{n+1})-
J_{n}(p_{1}+p_{2}, p_{3}, \ldots, p_{n+1})-
(p_{1}+p_{2})^{\alpha}J_{2}\left(\frac{p_{1}}{p_{1}+p_{2}}, \frac{p_{2}}{p_{1}+p_{2}}\right) \\
=I_{n+1}(p_{1}, \ldots, p_{n+1})-I_{n}(p_{1}+p_{2}, p_{3}, \ldots, p_{n+1})-
(p_{1}+p_{2})^{\alpha}I_{2}\left(\frac{p_{1}}{p_{1}+p_{2}}, \frac{p_{2}}{p_{1}+p_{2}}\right) \\
+I_{n}(p_{1}+p_{2},p_{3}, \ldots, p_{n+1})-J_{n}(p_{1}+p_{2},p_{3}, \ldots, p_{n+1}) \\
+ (p_{1}+p_{2})^{\alpha}\left(I_{2}\left(\frac{p_{1}}{p_{1}+p_{2}}\right)-
J_{2}\left(\frac{p_{1}}{p_{1}+p_{2}}, \frac{p_{2}}{p_{1}+p_{2}}\right)\right).
\end{array}
\]

Therefore, if $\alpha<0$,  (\ref{Eq3.6}) (with $n+1$ instead of $n$),
(\ref{Eq3.8}) with $n=2$ and the induction hypothesis
(applying to $(p_{1}+p_{2}, \ldots, p_{n+1})$ instead of $(p_{1}, \ldots, p_{n})$)
imply that

\begin{multline*}
\left|I_{n+1}(p_{1}, \ldots, p_{n+1})-J_{n+1}(p_{1}, \ldots, p_{n+1})\right| \\
\leq \varepsilon_{n}+\sum^{n-1}_{k=2}\varepsilon_{k}+
K(\alpha)(2\varepsilon_{2}+\varepsilon_{1})\left(1+\sum^{n-1}_{k=2}\left(\sum^{k+1}_{i=1}p_{i}^{\alpha}\right)\right) \\
+K(\alpha)(2\varepsilon_{2}+\varepsilon_{1})(p_{1}+p_{2})^{\alpha} \\
= \sum^{n}_{k=2}\varepsilon_{k}+K(\alpha)(2\varepsilon_{2}+\varepsilon_{1})
\left(1+\sum^{n}_{k=2}\left(\sum^{k}_{i=1}p_{i}^{\alpha}\right)\right),
\end{multline*}

that is (\ref{Eq3.8}) holds for $n+1$ instead of $n$.

Furthermore, if $\alpha=0$, (\ref{Eq3.6}) (with $n+1$ instead of $n$),
(\ref{Eq3.9}) with $n=2$ and the induction hypothesis
(applying to $(p_{1}+p_{2}, \ldots, p_{n+1})$ instead of $(p_{1}, \ldots, p_{n})$)
imply that

\begin{multline*}
\left|I_{n+1}(p_{1}, \ldots, p_{n+1})-J_{n+1}(p_{1}, \ldots, p_{n+1})\right| \\
\leq \varepsilon_{n}+\sum^{n-1}_{k=2}\varepsilon_{k}+K(\alpha)(n-1)(2\varepsilon_{2}+\varepsilon_{1})+
K(\alpha)(2\varepsilon_{2}+\varepsilon_{1}) \\
=\sum^{n}_{k=2}\varepsilon_{k}+K(\alpha)n(2\varepsilon_{2}+\varepsilon_{1}).
\end{multline*}

This yields that (\ref{Eq3.9}) holds for $n+1$ instead of $n$.

Finally, if $\alpha>0$, then (\ref{Eq3.6}) (with $n+1$ instead of $n$),
(\ref{Eq3.10}) with $n=2$ and the induction hypothesis
(applying to $(p_{1}+p_{2}, \ldots, p_{n+1})$ instead of $(p_{1}, \ldots, p_{n})$)
imply that

\begin{multline*}
\left|I_{n+1}(p_{1}, \ldots, p_{n+1})-J_{n+1}(p_{1}, \ldots, p_{n+1})\right| \\
\leq \varepsilon_{n}+\sum^{n-1}_{k=2}\varepsilon_{k}+K(\alpha)(n-1)(2\varepsilon_{2}+\varepsilon_{1})+
K(\alpha)(2\varepsilon_{2}+\varepsilon_{1}) \\
=\sum^{n}_{k=2}\varepsilon_{k}+K(\alpha)n(2\varepsilon_{2}+\varepsilon_{1}),
\end{multline*}

that is, (\ref{Eq3.10}) holds for $n+1$ instead of $n$.
\end{proof}

\begin{Rem}
Applying Theorem \ref{Thm3.2} with the choice $\varepsilon_{n}=0$
for all $n\in\mathbb{N}$, we get the $\alpha$--recursive, $3$--semi--symmetric
information measures. Hence Theorem \ref{Thm3.2} says that
the system of $\alpha$--recursive and $3$--semi--symmetric information measures is
stable.
\end{Rem}

\section{Open problems}

In the last part of the paper we list some open problems from the investigated topic.

The stability of equation (\ref{Eq1.1}) in the exceptional case $\alpha=1$
was raised by Sz\'{e}kelyhidi in \cite{Szek91}, and it is still open.

\begin{Opp}
Prove or disprove that the fundamental equation of information is stable on the
set $D^{\circ}$ or on the set $D$.
\end{Opp}

We remark that concerning this problem a partial result was published in
Morando \cite{Mor01}.

In the monograph of Ebanks, Sahoo and Sander (see \cite{ESS98})
higher dimensional information measures and functions are considered.
A stability type result was published in \cite{Gse08}, assuming the
underlying multiplicative function is bounded on its closed domain.
Therefore the following problem can be formulated.

\begin{Opp}
What can be said about the stability of the
fundamental equation of information of multiplicative type on the closed
as well as on the open domain?
\end{Opp}

In the inset theory (see e.g. Acz\'{e}l--Dar\'{o}czy \cite{AD75}),
measures of information may be depend
on both the probabilities and events.
Thus the problem of finding all
inset information measures lead to the generalized
fundamental equation of information of degree alpha, that is,
to the functional equation
\[
f(x)+(1-x)^{\alpha}g\left(\frac{y}{1-x}\right)=
h(x)+(1-y)^{\alpha}k\left(\frac{x}{1-y}\right).
\quad \left((x, y)\in D^{\circ}\right)
\]
This equation was solved in Maksa \cite{Mak82}.

\begin{Opp}
Is it true that the generalized fundamental equation of information
of degree alpha is stable on the set $D^{\circ}$?
\end{Opp}

\noindent
\textbf{Acknowledgement. }I am grateful to Professor Gyula Maksa for encouraging me
to write this paper and I would like to express my
sincere thanks for his help during the preparation
of the manuscript.

\end{document}